\documentclass[11pt,oneside,onecolumn]{amsart}
\usepackage{amssymb}
\usepackage{xcolor}
\usepackage{mathrsfs}
\numberwithin{equation}{section}
\allowdisplaybreaks[4]
\usepackage{amsfonts}
\newtheorem{Theorem}{Theorem}[section]
\newtheorem{Proposition}[Theorem]{Proposition}

\newtheorem{Lemma}[Theorem]{Lemma}

\def\ep{\varepsilon}

\def\bbZ{\mathbb{Z}}
\def\bbR{\mathbb{R}}

\begin{document}

\title[Multilinear Fractional Integrals]{Sharp Weighted Inequalities for Multilinear Fractional Maximal Operator and  Fractional Integrals}

\thanks{The second author is partially supported by the NSF under grant 1201504. The third author is
partially supported by the National Natural Science Foundation of China (11371200)  and the Research Fund for the Doctoral Program
of Higher Education(20120031110023).}

\author{Kangwei Li}

\address{School of Mathematical Sciences and LPMC,  Nankai University,
      Tianjin~300071, China}
\email{likangwei9@mail.nankai.edu.cn}

\author{Kabe Moen}

\address{Department of Mathematics, University of Alabama, Tuscaloosa, AL 35487, USA}
\email{kabe.moen@ua.edu}

\author{Wenchang Sun}

\address{School of Mathematical Sciences and LPMC,  Nankai University,
      Tianjin~300071, China}
\email{sunwch@nankai.edu.cn}

\begin{abstract}
In this paper, we study the  weighted inequality for multilinear fractional maximal operators and fractional integrals.
We give sharp weighted estimates for both operators.
\end{abstract}
\maketitle

\section{Introduction and Main Results}
Fractional type operators and associated maximal functions are useful tools  in harmonic analysis,
especially in the study of differentiability or smoothness properties of functions.
Recall that for $0<\alpha<n$, the fractional maximal function $M_\alpha$ and the fractional integral operator $I_\alpha$ are defined by
\[
  M_\alpha(f)(x):=\sup_{Q\ni x}\frac{1}{|Q|^{1-\alpha/n}}\int_Q |f(y)|dy
\]
and
\[
  I_\alpha (f)(x):=\int \frac{f(y)}{|x-y|^{n-\alpha}}dy,
\]
respectively, where $f$ is a locally integrable function defined on $\bbR^n$ and $Q$ is a cube in $\bbR^n$.
We refer to \cite{Gra1,Stein} for the basic
properties of these operators.

Let $w$ be a weight, i.e., a non-negative locally integrable function.
Muckenhoupt and Wheeden \cite{MW} showed that for $1<p<n/\alpha$ and $1/q = 1/p-\alpha/n$,
 $I_\alpha$ is bounded from $L^p(w^p)$ to $L^q(w^q)$ if and only if  $w$  belongs to the $A_{p,q}$ class, that is,
\[
  [w]_{A_{p,q}}:=\sup_{Q}\left(\frac{1}{|Q|}\int_Q w(x)^qdx\right)\left(\frac{1}{|Q|}\int_Q w(x)^{-p'}dx\right)^{q/{p'}}<\infty.
\]
Moreover, the fractional maximal function $M_\alpha$  is also  bounded from $L^p(w^p)$ to $L^q(w^q)$.
See also \cite{DL,ISST,LLL,SW,WTL,YY,YYZ,ZC} for more results of fractional integral operators on various function spaces.

In \cite{LMPT}, Lacey, Moen, P\'erez and Torres gave the sharp weighted estimates for both $M_\alpha$ and $I_\alpha$. Specifically, they proved that
\begin{equation}\label{eq:ee1}
\|M_\alpha\|_{L^p(w^p)\rightarrow L^q(w^q)}\le C_{n,p}[w]_{A_{p,q}}^{(1-\frac{\alpha}{n})\frac{p'}{q}}
\end{equation}
and
\begin{equation}\label{eq:ee2}
\|I_\alpha\|_{L^p(w^p)\rightarrow L^q(w^q)}\le C_{n,p}[w]_{A_{p,q}}^{(1-\frac{\alpha}{n})\max\{1,\frac{p'}{q}\}},
\end{equation}
where the exponents of $[w]_{A_{p,q}}$ are sharp.

Multilinear fractional integral operators were studied by many authors, e.g., Grafakos \cite{G}, Kenig and
Stein \cite{KS}, Grafakos and Kalton \cite{GK}, Chen and Xue \cite{CX}, Pradolini \cite{P},
Kuk and Lee \cite{KL},
 and Mei, Xue and Lan \cite{MXL}.
For $\vec f = (f_1,\cdots,f_m)$ and $\quad 0<\alpha<mn$,  the multilinear fractional maximal function $\mathcal{M_\alpha}$ and the multilinear integral operator
$ \mathcal{I}_\alpha $ are defined by
\[
  \mathcal{M}_\alpha (\vec f)(x)=\sup_{Q\ni x}\prod_{i=1}^m\frac{1}{|Q|^{1-\alpha/{mn}}}\int_Q |f_i(y_i)|dy_i
\]
and
\[
  \mathcal{I}_\alpha (\vec{f})(x)=\int_{\bbR^{mn}} \frac{f_1(y_1)\cdots f_m(y_m)}{(|x-y_1|+\cdots+|x-y_m|)^{mn-\alpha}}d\vec{y},
\]
respectively.

To study the weighted estimates for  multilinear fractional integral operators,
Moen \cite{M} introduced the multiple $A_{\vec{P},q}$ weight.
Let $1/p_1 + \cdots + 1/p_m = 1/q + \alpha/n$.
A multiple weight $(w_1,\cdots, w_m)$ is said
to belong to the $A_{\vec{P},q}$ class if and only if
\[
 [\vec{w}]_{A_{\vec{P},q}}:= \sup_{Q}\left(\frac{1}{|Q|}\int_Q u_{\vec{w}}\right)\prod_{i=1}^m \left(\frac{1}{|Q|}\int_Q \sigma_i\right)^{q/{p_i'}}<\infty,
\]
where $u_{\vec{w}}=(\Pi_{i=1}^m w_i)^q$ and $\sigma_i=w_i^{-p_i'}$.
Moen showed that
$\mathcal{I}_\alpha$ is bounded from $L^{p_1}(w_1^{p_1})\times\cdots\times L^{p_m}(w_m^{p_m})$
to $L^q(u_{\vec{w}})$ if and only if $ [\vec{w}]_{A_{\vec{P},q}}<\infty$.

In this paper, we study the sharp estimates for both $\mathcal{M}_\alpha$ and $\mathcal{I}_\alpha$.
For the multilinear fractional maximal function, we give a sharp estimate for $0<\alpha<n$.

\begin{Theorem}\label{thm:m1}
Suppose that $0<\alpha<n$,
$q>0$, $1<p_1,\cdots,p_m<\infty$,
$1/p_1+\cdots +1/p_m=1/q+\alpha/n$,
$p'_{i_0}=\max\{p'_i:\,1\le i\le m\}$ for some $i_0$ and
$p'_{i_0}(1-\alpha/n)\ge \max_{i\ne i_0}\{p_i'\}$. Let
$\vec{w}=(w_1,\cdots,w_m)\in A_{\vec{P},q}$. Then we have
\begin{equation}\label{eq:thm:m1:e1}
  \|\mathcal{M}_{\alpha}(\vec{f})\|_{L^q(u_{\vec{w}})}\le
  C_{m,n,\vec{P},q}[\vec{w}]_{A_{\vec{P},q}}^{(1-\frac\alpha n)
  \frac{\max\{p'_i\}}{q}}\prod_{i=1}^m \|f_i\|_{L^{p_i}(w_i^{p_i})},
\end{equation}
where the exponent of $[\vec{w}]_{A_{\vec{P},q}}$ is sharp.
\end{Theorem}

And for the multilinear fractional integral operators, we also get a sharp estimate in some cases.
\begin{Theorem}\label{thm:m2}
Suppose that $0<\alpha< n$, $q>0$, $1<p_1,\cdots,p_m<\infty$, $1/p_1+\cdots +1/p_m=1/q+\alpha/n$,  $\vec{w}=(w_1,\cdots,w_m)\in A_{\vec{P},q}$ and
\[
  \min\left\{\frac{\max_i{p_i'}}{q},
   \min_j\frac{\max_{i\neq j}\{p_i', q\}}{p_j'}\right\}\le 1-\frac{\alpha}{n}.
\]
Then
\[
  \|\mathcal{I}_\alpha(\vec{f})\|_{L^q(u_{\vec{w}})}\le
  C_{m,n,\vec{P},q}[\vec{w}]_{A_{\vec{P},q}}^{(1-\frac\alpha n)\max_i\{1,\frac{p_i'}{q}\}}\prod_{i=1}^m \|f_i\|_{L^{p_i}(w_i^{p_i})},
\]
where the exponent of $[\vec{w}]_{A_{\vec{P},q}}$ is sharp.
\end{Theorem}

Similar to the estimate for Calder\'on-Zygmund operators,
although the estimate in (\ref{eq:thm:m1:e1}) is sharp, it can be improved whenever mixed estimates are invoked.

\begin{Theorem}\label{thm:m3}
Suppose that $0\le \alpha< mn$,  $q>0$, $1<p_1,\cdots,p_m<\infty$, $1/p_1+\cdots +1/p_m=1/q+\alpha/n$ and  $\vec{w}=(w_1,\cdots,w_m)\in A_{\vec{P},q}$. Then
\[
    \|\mathcal{M}_{\alpha}(\vec{f})\|_{L^q(u_{\vec{w}})}\le
  C_{m,n,\vec{P},q}\left([\vec{w}]_{A_{\vec{P},q}}^{\frac{1}{q}}\prod_{i=1}^m [\sigma_i]_{A_\infty}^{\frac{1}{p_i}(1-\frac{\alpha p}{n})}\right)\prod_{i=1}^m \|f_i\|_{L^{p_i}(w_i^{p_i})},
\]
where the exponents are sharp.
\end{Theorem}

Moreover, in the so-called two weight case -- dropping the requirement that $u_{\vec w}=\big(\Pi_i w_i\big)^q$ and taking $u$ to be independent of $\vec w$ -- one can break the homogeneity requirement that $\frac1q=\frac1p-\frac{\alpha}{n}$ and simply require $q\geq p$.  If we are given an $n+1$ tuple of weights $(u,w_1,\ldots,w_n)=(u,\vec w)$ then define
$$[u,\vec w]_{A_{\vec P,q}}=
\sup_Q|Q|^{\frac{\alpha}{n}+\frac1q-\frac1p}\left(\frac{1}{|Q|}\int_Q u\right)^{1/q}\prod_{i=1}^m \left(\frac{1}{|Q|}\int_Q \sigma_i\right)^{1/{p_i'}},$$
where, again, $\sigma_i=w_i^{-p_i'}.$
Then we have the following extension of Theorem \ref{thm:m3}.

\begin{Theorem}\label{thm:m4}
Suppose that $0\le \alpha< mn$, $1<p_1,\cdots,p_m<\infty$, $1/p_1+\cdots +1/p_m=1/p$ and $p\leq q$. If $(u,\vec{w})$ is an $m+1$ tuple of weights satisfying $\sigma_i=w_i^{-p_i'}\in A_\infty$ for all $i=1,\ldots,m$ and $[u,\vec{w}]_{A_{\vec{P},q}}<\infty$, then
\[
    \|\mathcal{M}_{\alpha}(\vec{f})\|_{L^q(u)}\le
  C_{m,n,\vec{P},q}\left([u,\vec{w}]_{A_{\vec{P},q}}\prod_{i=1}^m [\sigma_i]_{A_\infty}^{\frac{p}{q}\frac{1}{p_i}}\right)\prod_{i=1}^m \|f_i\|_{L^{p_i}(w_i^{p_i})}
\]
\end{Theorem}

Theorem \ref{thm:m4} is the multilinear version of Theorem 2.2 in \cite{CM2}.  To see that Theorem \ref{thm:m4} does indeed extend Theorem \ref{thm:m3} notice that when $\frac{1}{q}=\frac{1}{p}-\frac{\alpha}{n}$ we have $\frac{p}{q}=1-\frac{\alpha p}{n}$ and if $u=u_{\vec w}=\big(\Pi_i w_i\big)^q$ then
$$[u_{\vec{w}},\vec w]_{A_{\vec P,q}}=[\vec w]_{A_{\vec{P},q}}^{1/q}.$$
 Even though Theorem \ref{thm:m3} follows from Theorem \ref{thm:m4} we provide proofs for both Theorems.  We present both proofs because they use different techniques.  For Theorem \ref{thm:m3} we use a sharp reverse H\"older estimate (see Proposition \ref{prop:reverseH} below), while Theorem \ref{thm:m4} will use Carleson sequences (see Lemma \ref{lem:carleson}).

\section{Proofs of Theorem~\ref{thm:m1} and Theorem~\ref{thm:m2}}
Recall that the standard dyadic grid in $\bbR^n$ consists of the cubes
\[
  [0,2^{-k})^n+2^{-k}j,\qquad k\in\bbZ, j\in\bbZ^n.
\]
Denote the standard dyadic grid by $\mathcal{D}$.

By a general dyadic grid $\mathscr{D}$ we mean a collection of cubes with the following
properties: (i) for any $Q\in\mathscr{D}$ its sidelength $l_Q$ is of the form $2^k$, $k\in\bbZ$; (ii)
 $Q\cap R \in \{Q,R,\varnothing\}$ for any $Q,R\in\mathscr{D}$; (iii) the cubes of a fixed sidelength $2^k$ form a partition
of $\bbR^n$.

For any $0\le\alpha<n$, we define
\[
  M_{\alpha,w}^{\mathscr{D}}f(x)=\sup_{Q\ni x, Q\in \mathscr{D}}\frac{1}{w(Q)^{1-\frac{\alpha}{n}}}\int_Q |f|w.
\]
When $\alpha=0$, we denote $M_{0,w}^{\mathscr{D}}$ simply by $M_{w}^{\mathscr{D}}$.
We need the following result.
\begin{Proposition}[{\cite[Theorem 2.3]{M1}}]\label{prop:p2}
If $0\le \alpha<n$, $1<p\le \frac n\alpha$ and $1/q=1/p-\alpha/n$, then
\[
  \|M_{\alpha,w}^{\mathscr{D}}f\|_{L^q(w)}\le (1+\frac{p'}{q})^{1-\frac \alpha n}\|f\|_{L^p(w)}.
\]
\end{Proposition}

We say that  $\mathcal{S}:=\{Q_{j,k}\}$ is a sparse family of cubes if:
\begin{enumerate}
\item for each fixed $k$ the cubes $Q_{j,k}$ are pairwise disjoint;
\item if $\Gamma_k=\bigcup_j Q_{j,k}$, then $\Gamma_{k+1}\subset \Gamma_k$;
\item $|\Gamma_{k+1}\bigcap Q_{j,k}|\le \frac{1}{2}|Q_{j,k}|$.
\end{enumerate}
For any $Q_{j,k}\in\mathcal{S}$, we define $E(Q_{j,k})=Q_{j,k}\setminus \Gamma_{k+1}$.
Then the sets $E(Q_{j,k})$ are pairwise disjoint and $|E(Q_{j,k})|\ge \frac{1}{2}|Q_{j,k}|$.

Define
\[
  \mathscr{D}_t:=\{2^{-k}([0,1)^n+m+(-1)^k t): k\in\bbZ, m\in\bbZ^n\},\quad t\in\{0, 1/3\}^n.
\]
The importance of these grids is shown by the following proposition, which
 can be found in \cite[proof of Theorem 1.10]{HP}. See also \cite[Proposition 5.1]{L}.
\begin{Proposition}\label{prop:p1}
There are $2^n$ dyadic grids $\mathscr{D}_t$, $t\in\{0,1/3\}^n$ such that for any cube
$Q\subset \bbR^n$ there exists a cube $Q_t \in \mathscr{D}_t$ satisfying
 $Q \subset Q_t$ and $l(Q_t)\le 6l(Q)$.
\end{Proposition}

Given a dyadic grid $\mathscr{D}$ and  a sparse family $\mathcal{S}$ in $\mathscr{D}$.
Define the dyadic fractional integral operators by
\[
  (\mathcal{I}_\alpha^{\mathcal{S}})(\vec{f})(x)
  =\sum_{Q\in\mathcal{S}}|Q|^{\frac{\alpha}{n}-m}
   \prod_{i=1}^m\int_Q f_i(y_i)dy_i \cdot\chi_Q(x)
\]
and
\[
  (\mathcal{I}_{\alpha,q}^{\mathcal{S}})(\vec{f})(x)
  =\left(\sum_{Q\in\mathcal{S}}\left(|Q|^{\frac{\alpha}{n}-m}
   \prod_{i=1}^m\int_Q f_i(y_i)dy_i \right)^q\cdot\chi_Q(x)\right)^{1/q}.
\]

In \cite{M}, Moen showed that for $q> 1$
\[
  (\mathcal{I}_\alpha(\vec{f}))(x)
   \le C\sum_{Q\in\mathscr{D}}
   \frac{l(Q)^\alpha}{|3Q|^m}\int_{(3Q)^m} f_1(y_1)\cdots f_m(y_m)d\vec{y}\chi^{}_Q(x),
\]
and for $q\le1$,
\[
  (\mathcal{I}_\alpha(\vec{f}))(x)
  \le C\left(\sum_{Q\in\mathscr{D}}\left(\frac{l(Q)^\alpha}{|3Q|^m}
   \int_{(3Q)^m} f_1(y_1)\cdots f_m(y_m)d\vec{y}\right)^{q}\chi^{}_Q(x)\right)^{1/q}.
\]

By Proposition~\ref{prop:p1} and similar arguments as that in \cite{CM}, we get
\begin{equation}\label{eq:q Gt 1}
  \mathcal{I}_\alpha(\vec{f})(x)\lesssim
   \sum_{t\in\{0,1/3\}^n}(\mathcal{I}_\alpha^{\mathcal{S}_t})(\vec{f})(x),\quad q> 1,
\end{equation}
and
\begin{equation}\label{eq:q Lt 1}
  \mathcal{I}_\alpha(\vec{f})(x)\lesssim
   \sum_{t\in\{0,1/3\}^n}(\mathcal{I}_{\alpha,q}^{\mathcal{S}_t})(\vec{f})(x),\quad q\le1.
\end{equation}
We have the following lemmas.
\begin{Lemma} \label{Lm:dual}
Suppose that $\vec{w}=(w_1,\cdots,w_m)\in A_{\vec{P},q}$,
$0\le\alpha<mn$, $1<q,p_1,\cdots,p_m<\infty$ and $1/p_1+\cdots +1/p_m=1/q+\alpha/n$.
Then
\[
  \vec{w}^i:=(w_1,\cdots,w_{i-1},(\Pi_{i=1}^m w_i)^{-1},w_{i+1},\cdots,w_m)\in A_{\vec{P}^i,p_i'},
\]
where $\vec{P}^i=(p_1,\cdots,p_{i-1},q',p_{i+1},\cdots,p_m)$ and $[\vec{w}^i]_{A_{\vec{P}^i,p_i'}}=[\vec{w}]_{A_{\vec{P},q}}^{p_i'/q}$.
\end{Lemma}
\begin{proof}
By the definition we have
\begin{eqnarray*}
[\vec{w}^i]_{A_{\vec{P}^i,p_i'}}&=&\sup_Q \left(\frac{1}{|Q|}\int_Q w_i^{-p_i'}\right)\left(\frac{1}{|Q|}\int_Q u_{\vec{w}}\right)^{p_i'/q}
\prod_{j\neq i}\left(\frac{1}{|Q|}\int_Q w_j^{-p_j'}\right)^{p_i'/{p_j'}}\\
&=&[\vec{w}]_{A_{\vec{P},q}}^{p_i'/q}.
\end{eqnarray*}
\end{proof}

\begin{Lemma}\label{lm:l1}
Suppose that $0<\alpha< n$,  $1<p_1,\cdots,p_m<\infty$,
$1/p_1+\cdots +1/p_m=1/q+\alpha/n$, $q(1-\alpha/n)\ge max\{p_i'\}$
 and
that $\vec{w}=(w_1,\cdots,w_m)\in A_{\vec{P},q}$. Then
\[
  \|\mathcal{I}_\alpha^{\mathcal{S}}(|f_1|,\cdots,|f_m|)\|_{L^q(u_{\vec{w}})}\le
  C_{m,n,\vec{P},q}[\vec{w}]_{A_{\vec{P},q}}^{1-\alpha/n}\prod_{i=1}^m \|f_i\|_{L^{p_i}(w_i^{p_i})}.
\]
\end{Lemma}
\begin{proof}
It is equivalent to prove the following
\[
  \|\mathcal{I}_\alpha^{\mathcal{S}}(|f_1|\sigma_1,\cdots,|f_m|\sigma_m)\|_{L^q(u_{\vec w})}\le
  C_{m,n,\vec{P},q}[\vec{w}]_{A_{\vec{P},q}}^{1-\alpha/n}\prod_{i=1}^m \|f_i\|_{L^{p_i}(\sigma_i)}.
\]
Let $1/p=1/q+\alpha/n$. We have
\begin{eqnarray}
  \label{eq:Lm:L1:e1}
&&\|\mathcal{I}_\alpha^{\mathcal{S}}(|f_1|\sigma_1,\cdots,|f_m|\sigma_m)\|_{L^q(u_{\vec w})}
   \\
&=&\sup_{\|h\|_{L^{q'}(u_{\vec w})}=1}\int_{\bbR^n}\mathcal{I}_\alpha^{\mathcal{S}}
  (|f_1|\sigma_1,\cdots,|f_m|\sigma_m)h u_{\vec w}\nonumber \\
&=&\sup_{\|h\|_{L^{q'}(u_{\vec w})}=1}\sum_{Q\in\mathcal{S}}|Q|^{\frac{\alpha}{n}-m}
  \prod_{i=1}^m\int_Q |f_i|\sigma_i dy_i \cdot\int_Q h u_{\vec w}\nonumber \\
&=&\sup_{\|h\|_{L^{q'}(u_{\vec w})}=1}\sum_{Q\in\mathcal{S}}
   \left(\frac{u_{\vec w}(Q)}{|Q|}\right)^{1-\frac{\alpha}{n}}
\prod_{i=1}^m\left(\frac{\sigma_i(Q)}{|Q|}\right)^{\frac{q(1-\alpha/n)}{p_i'}}
\prod_{i=1}^m \sigma_i(Q)^{1-\frac{q}{p_i'}(1-\frac{\alpha}{n})} \nonumber \\
&&\times|Q|^{(m-\frac{\alpha}{n})((1-\frac{\alpha}{n})q-1)}
  \cdot\frac{1}{u_{\vec w}(Q)^{1-\alpha/n}}\int_Q h u_{\vec w}\cdot
\prod_{i=1}^m\frac{1}{\sigma_i(Q)}\int_Q |f_i|\sigma_i dy_i\nonumber \\
&\le&[\vec{w}]_{A_{\vec{P},q}}^{1-\alpha/n}\sup_{\|h\|_{L^{q'}(u_{\vec w})}=1}
  \sum_{Q\in\mathcal{S}}
|Q|^{(m-\frac{\alpha}{n})\frac{q}{p'}}
\prod_{i=1}^m \sigma_i(Q)^{1-\frac{q}{p_i'}(1-\frac{\alpha}{n})}\nonumber  \\
&&\times\frac{1}{u_{\vec w}(Q)^{1-\alpha/n}}\int_Q h u_{\vec w}\cdot
  \prod_{i=1}^m\frac{1}{\sigma_i(Q)}\int_Q |f_i|\sigma_i dy_i. \nonumber
\end{eqnarray}
By H\"older's inequality, we have
\begin{eqnarray}
\label{eq:e4}
|E(Q)|&=&\int_{E(Q)} u_{\vec w}^{\frac{1}{(m-\frac{\alpha}{n})q}}\prod_{i=1}^m \sigma_i^{\frac{1}{(m-\frac{\alpha}{n})p_i'}}dx\\
&\le&u_{\vec w}(E(Q))^{\frac{1}{(m-\frac{\alpha}{n})q}}\prod_{i=1}^m \sigma_i(E(Q))^{\frac{1}{(m-\frac{\alpha}{n})p_i'}}.\nonumber
\end{eqnarray}
Since $|Q|\le 2|E(Q)|$ and $q(1-\alpha/n)\ge \max\{p_i'\}$, we have
\begin{eqnarray*}
&&|Q|^{(m-\frac{\alpha}{n})\frac{q}{p'}}
\prod_{i=1}^m \sigma_i(Q)^{1-\frac{q}{p_i'}(1-\frac{\alpha}{n})}\\
&\lesssim& u_{\vec w}(E(Q))^{\frac{1}{p'}}\prod_{i=1}^m
  \sigma_i(E(Q))^{\frac{q}{p'p_i'}}\prod_{i=1}^m
     \sigma_i(Q)^{1-\frac{q}{p_i'}(1-\frac{\alpha}{n})}\\
&\le&u_{\vec w}(E(Q))^{\frac{1}{p'}}\prod_{i=1}^m \sigma_i(E(Q))^{1/{p_i}}.
\end{eqnarray*}
Hence
\begin{eqnarray*}
&&\sum_{Q\in\mathcal{S}}|Q|^{(m-\frac{\alpha}{n})\frac{q}{p'}}
\prod_{i=1}^m \sigma_i(Q)^{1-\frac{q}{p_i'}(1-\frac{\alpha}{n})}
  \frac{1}{u_{\vec w}(Q)^{1-\alpha/n}}\int_Q h u_{\vec w}\\
&&\times\prod_{i=1}^m\frac{1}{\sigma_i(Q)}\int_Q |f_i|\sigma_i dy_i\\
&\lesssim& \sum_{Q\in\mathcal{S}}
\frac{1}{u_{\vec w}(Q)^{1-\alpha/n}}\int_Q h u_{\vec w}
  \cdot u_{\vec w}(E(Q))^{\frac{1}{p'}}\cdot\prod_{i=1}^m
    \frac{1}{\sigma_i(Q)}\int_Q |f_i|\sigma_i dy_i\\
&&\times\prod_{i=1}^m \sigma_i(E(Q))^{1/{p_i}}\\
&\le&\left(\sum_{Q\in\mathcal{S}} \left(\frac{1}{u_{\vec w}(Q)^{1-\alpha/n}}
  \int_Q h u_{\vec w}\right)^{p'}u_{\vec w}(E(Q))\right)^{1/{p'}}\\
&&\times \prod_{i=1}^m \left(\sum_{Q\in\mathcal{S}} \left(\frac{1}{\sigma_i(Q)}
  \int_Q |f_i|\sigma_i dy_i\right)^{p_i}\sigma_i(E(Q))\right)^{1/{p_i}}\\
&\le&\|M_{\alpha,u_{\vec w}}^{\mathscr{D}}h\|_{L^{p'}(u_{\vec w})}\prod_{i=1}^m
  \|M_{\sigma_i}^{\mathscr{D}}f_i\|_{L^{p_i}(\sigma_i)}\\
&\lesssim& \|h\|_{L^{q'}(u_{\vec w})}\prod_{i=1}^m \|f_i\|_{L^{p_i}(\sigma_i)}.
\end{eqnarray*}
It follows from (\ref{eq:Lm:L1:e1}) that
\[
  \|\mathcal{I}_\alpha^{\mathcal{S}}(|f_1|\sigma_1,\cdots,|f_m|\sigma_m)\|_{L^q(u_{\vec w})}\le
  C_{m,n,\vec{P},q}[\vec{w}]_{A_{\vec{P},q}}^{1-\alpha/n}\prod_{i=1}^m \|f_i\|_{L^{p_i}(\sigma_i)}.
\]
\end{proof}

\begin{Lemma}\label{lm:lq}
Let the hypotheses be as in Theorem~\ref{thm:m1}.
Moreover, let $\mathcal S$ be a sparse family of cubes.
Then we have
\[
  \|\mathcal{I}_{\alpha,q}^{\mathcal{S}}(|f_1|,\cdots,|f_m|)\|_{L^q(u_{\vec w})}\le
  C_{m,n,\vec{P},q}[\vec{w}]_{A_{\vec{P},q}}^{(1-\frac\alpha n)\max_i\{\frac{p_i'}{q}\}}\prod_{i=1}^m \|f_i\|_{L^{p_i}(w_i^{p_i})}.
\]
\end{Lemma}
\begin{proof}
Without loss of generality, assume that $p_1=\min\{p_1,\cdots,p_m\}$.
As in Lemma~\ref{lm:l1},
it is equivalent to prove the following
\[
  \|\mathcal{I}_{\alpha,q}^{\mathcal{S}}
   (|f_1|\sigma_1,\cdots,|f_m|\sigma_m)\|_{L^q(u_{\vec w})}
   \!\le\!
  C_{m,n,\vec{P},q}[\vec{w}]_{A_{\vec{P},q}}^{(1-\alpha/n)\max\{p'_i/q\}}
   \prod_{i=1}^m \|f_i\|_{L^{p_i}(\sigma_i)}.
\]
We have
\begin{eqnarray*}
&&\|\mathcal{I}_{\alpha,q}^{\mathcal{S}}(|f_1|\sigma_1,\cdots,|f_m|\sigma_m)
   \|_{L^q(u_{\vec w})}^q\\
&=&\sum_{Q\in\mathcal{S}}\left(|Q|^{\frac{\alpha}{n}-m}\prod_{i=1}^m\int_Q
  |f_i|\sigma_i dy_i \right)^q\cdot u_{\vec w}(Q)\\
&=&\sum_{Q\in\mathcal{S}}\left(\frac{u_{\vec w}(Q)}{|Q|}\right)^{(1-\frac{\alpha}{n})p_1'}
\prod_{i=1}^m\left(\frac{\sigma_i(Q)}{|Q|}\right)^{\frac{(1-\alpha/n)qp_1'}{p_i'}}
  |Q|^{q(m-\frac{\alpha}{n})((1-\frac{\alpha}{n})p_1'-1)}\\
&&\times\prod_{i=2}^m \sigma_i(Q)^{q-\frac{qp_1'}{p_i'}(1-\frac{\alpha}{n})}
    u_{\vec w}(Q)^{1-(1-\alpha/n)p_1'}\left(\frac{\sigma_1(Q)^{\frac{\alpha}{n}}}{\sigma_1(Q)}
       \int_Q |f_1|\sigma_1 dy_1 \right)^{q}\\
&&\times\prod_{i=2}^m \left(\frac{1}{\sigma_i(Q)}\int_Q |f_i|\sigma_i dy_i \right)^{q}\\
&\lesssim&[\vec{w}]_{A_{\vec{P},q}}^{(1-\alpha/n)p_1'}\sum_{Q\in\mathcal{S}}
  \left(\frac{\sigma_1(Q)^{\frac{\alpha}{n}}}{\sigma_1(Q)}\int_Q |f_1|\sigma_1 dy_1
    \right)^{q}\prod_{i=2}^m \left(\frac{1}{\sigma_i(Q)}\int_Q |f_i|\sigma_i dy_i
       \right)^{q}\\
&&\times|E(Q)|^{q(m-\frac{\alpha}{n})((1-\frac{\alpha}{n})p_1'-1)}
   u_{\vec{w}}(Q)^{1-(1-\alpha/n)p_1'}
\prod_{i=2}^m \sigma_i(Q)^{q-\frac{qp_1'}{p_i'}(1-\frac{\alpha}{n})}\\
&\le& [\vec{w}]_{A_{\vec{P},q}}^{(1-\alpha/n)p_1'}\sum_{Q\in\mathcal{S}}
  \left(\frac{\sigma_1(Q)^{\frac{\alpha}{n}}}{\sigma_1(Q)}\int_Q |f_1|
    \sigma_1 dy_1 \right)^{q}\prod_{i=2}^m \left(\frac{1}{\sigma_i(Q)}
      \int_Q |f_i|\sigma_i dy_i \right)^{q}\\
&&\times\sigma_1(E(Q))^{q((1-\frac{\alpha}{n})-\frac{1}{p_1'})}\prod_{i=2}^m
  \sigma_i(E(Q))^{q/{p_i}},
\end{eqnarray*}
where we use (\ref{eq:e4}) and the fact that $p_1'(1-\alpha/n) \ge \max_{i\neq 1}\{p_i'\}$
in the last step.

Let $1/q_1={1}/{p_1}-{\alpha}/{n}$.
By H\"older's inequality, we have
\begin{eqnarray*}
&&\|\mathcal{I}_{\alpha,q}^{\mathcal{S}}(|f_1|\sigma_1,\cdots,|f_m|\sigma_m)
   \|_{L^q(u_{\vec w})}^q\\
&\le& [\vec{w}]_{A_{\vec{P},q}}^{(1-\alpha/n)p_1'}
  \left(\sum_{Q\in\mathcal{S}}\left(\frac{\sigma_1(Q)^{\alpha/n}}{\sigma_1(Q)}
     \int_Q |f_1|\sigma_1 dy_1 \right)^{q_1}\sigma_1(E(Q))\right)^{q/{q_1}}\\
&&\times\prod_{i=2}^m\left(\sum_{Q\in\mathcal{S}}\left(\frac{1}{\sigma_i(Q)}
  \int_Q |f_i|\sigma_i dy_i \right)^{p_i}\sigma_i(E(Q))\right)^{q/{p_i}}\\
&\le&[\vec{w}]_{A_{\vec{P},q}}^{(1-\alpha/n)p_1'}\|M_{\alpha,\sigma_1}^{\mathscr{D}}
   f_1\|_{L^{q_1}(\sigma)}^q \prod_{i=2}^m
     \|M_{\sigma_i}^{\mathscr{D}}f_i\|_{L^{p_i}(\sigma_i)}^q
 \\
&\lesssim&[\vec{w}]_{A_{\vec{P},q}}^{(1-\alpha/n)p_1'}\prod_{i=1}^m
   \|f_i\|_{L^{p_i}(\sigma_i)}^{q},
\end{eqnarray*}
where $p_1'(1-\frac \alpha n)>1$ ensure that $p_1<\frac n \alpha$.
\end{proof}

\begin{proof}[Proof of Theorem~\ref{thm:m1}]
By Proposition~\ref{prop:p1}, we have
\[
  \mathcal{M}_{\alpha}(f_1,\cdots, f_m)(x)
  \le C_{m,n}\sum_{t\in\{0,1/3\}^n}\mathcal{M}_{\alpha}^{\mathscr{D}_t}
    (f_1,\cdots,f_m)(x).
\]
So it suffices to prove the desired conclusion for $\mathcal{M}_{\alpha}^{\mathscr{D}}$.
Let $a=2^{(m-\alpha/n)(n+1)}$  and $\Omega_k=\{x\in\bbR^n:
 \mathcal{M}_{\alpha}^{\mathscr{D}}(f_1,\cdots,f_m)(x)>a^k\}$. Suppose that
$\Omega_k=\bigcup_j Q_j^k$, where $Q_j^k$ are pairwise disjoint maximal dyadic cubes
 in $\Omega_k$. Then $\mathcal S:=\{Q_j^k\}$ is a sparse family.
To see this, it suffices to prove that
\[
  | Q_j^k\bigcap \Omega_{k+1}| \le \frac{1}{2} |Q_j^k|.
\]
In fact, by the definition of $Q_j^k$, we have
\[
  a^k < \prod_{i=1}^m \frac{1}{|Q_j^k|^{1-\alpha/mn}} \int_{Q_j^k} |f_i|
    \le 2^{mn-\alpha} a^k.
\]
It follows that
\begin{eqnarray*}
  | Q_j^k\bigcap \Omega_{k+1}|
&=& \sum_{Q_l^{k+1}\subset Q_j^k} |Q_l^{k+1}| \\
&\le& \sum_{Q_l^{k+1}\subset Q_j^k}  a^{-(k+1)/(m-\alpha/n)} \prod_{i=1}^m
   \Big(\int_{Q_l^{k+1}} f_i\Big)^{1/(m-\alpha/n)} \\
&\le& a^{-(k+1)/(m-\alpha/n)}\prod_{i=1}^m
\!\left(\sum_{Q_l^{k+1}\subset Q_j^k}
   \!\!\!\Big(\int_{Q_l^{k+1}} f_i\Big)^{1/(1-\alpha/nm)}\!\right)^{1/m}\! \\
&\le& a^{-(k+1)/(m-\alpha/n)}\prod_{i=1}^m
\left(  \int_{Q_j^k} f_i\right)^{1/(m-\alpha/n)} \\
&\le&
  a^{-1/(m-\alpha/n)} 2^n  |Q_j^k|
   \\
&=& \frac{1}{2} |Q_j^k|.
\end{eqnarray*}
Hence $\mathcal S:=\{Q_j^k\}$ is a sparse family.
Therefore,
\begin{eqnarray}
\label{eq:e1}&&\int_{\bbR^n}\mathcal{M}_{\alpha}^{\mathscr{D}}(f_1,\cdots,f_m)^q
  u_{\vec w}dx\\
&=& \sum_k \int_{\Omega_k\setminus \Omega_{k+1}}\mathcal{M}_{\alpha}^{\mathscr{D}}
  (f_1,\cdots,f_m)^q u_{\vec w}dx\nonumber\\
&\le& a^q\sum_{k,j}\left(\prod_{i=1}^m\frac{1}{|Q_j^k|^{1-\alpha/{mn}}}\int_{Q_j^k}
  |f_i(y_i)| dy_i\right)^q u_{\vec w}(E(Q_j^k))\nonumber\\
&\lesssim& \int_{\bbR^n}(\mathcal{I}_{\alpha,q}^{\mathcal{S}})(f_1,\cdots,
    f_m)^q u_{\vec w}dx\nonumber
\end{eqnarray}
Now the desired conclusion follows from Lemma~\ref{lm:lq}.
\end{proof}

\begin{proof}[Proof of Theorem~\ref{thm:m2}]
There are two cases.

(i).\,\, $q> 1$.  By (\ref{eq:q Gt 1}),
it suffices to prove that
\[
  \|\mathcal{I}_\alpha^{\mathcal{S}}(\vec{f})\|_{L^q(u_{\vec w})}\le
  C_{m,n,\vec{P},q}[\vec{w}]_{A_{\vec{P},q}}^{(1-\frac\alpha n)\max_i\{1,\frac{p_i'}{q}\}}\prod_{i=1}^m \|f_i\|_{L^{p_i}(w_i^{p_i})}.
\]
If $q(1-\alpha/n)\ge max\{p_i'\}$, the desired conclusion follows from Lemma~\ref{lm:l1}. If $p_j'(1-\frac{\alpha}{n})\ge \max_{i\neq j}\{p_i', q\}$,
without loss of generality, assume that $p_1'(1-\frac{\alpha}{n})\ge \max_{i\neq 1}\{p_i', q\}$.
By duality, we have
\begin{eqnarray*}
&& \|\mathcal{I}_\alpha^{\mathcal{S}}\|_{L^{p_1}(w_1^{p_1})\times\cdots\times L^{p_m}(w_m^{p_m})\rightarrow L^q(u_{\vec w})}\\
  &=&\|\mathcal{I}_\alpha^{\mathcal{S}}\|_{L^{q'}(\Pi_{i=1}^m w_i^{-q'})\times L^{p_2}(w_2^{p_2})\times\cdots
  \times L^{p_m}(w_m^{p_m})\rightarrow L^{p_1'}( w_1^{-p_1'})}\\
  &\lesssim&[\vec{w}]_{A_{\vec{P},q}}^{(1-\frac{\alpha}{n})\frac{p_1'}{q}},
\end{eqnarray*}
where we use Lemma~\ref{Lm:dual} in the last step.

(ii).\, $q\le1$. In this case,
\[
   \min_j \{\frac{\max_{i\neq j}\{p_i'\}}{p_j'}\}
 = \min\left\{\frac{\max_i{p_i'}}{q},
   \min_j\frac{\max_{i\neq j}\{p_i', q\}}{p_j'}\right\}
  \le 1-\frac{\alpha}{n}.
\]
By (\ref{eq:q Lt 1}),
it suffices to prove that
\[
  \|\mathcal{I}_{\alpha,q}^{\mathcal{S}}(\vec{f})\|_{L^q(u_{\vec w})}\le
  C_{m,n,\vec{P},q}[\vec{w}]_{A_{\vec{P},q}}^{(1-\frac\alpha n)\max_i\{1,\frac{p_i'}{q}\}}\prod_{i=1}^m \|f_i\|_{L^{p_i}(w_i^{p_i})}.
\]
By Lemma~\ref{lm:lq}, we get
\begin{eqnarray*}
\|\mathcal{I}_{\alpha,q}^{\mathcal{S}}(\vec{f})\|_{L^q(u_{\vec w})}&\le&
  C_{m,n,\vec{P},q}[\vec{w}]_{A_{\vec{P},q}}^{(1-\frac\alpha n)\max_i\{\frac{p_i'}{q}\}}\prod_{i=1}^m \|f_i\|_{L^{p_i}(w_i^{p_i})}\\
  &=&C_{m,n,\vec{P},q}[\vec{w}]_{A_{\vec{P},q}}^{(1-\frac\alpha n)\max_i\{1,\frac{p_i'}{q}\}}\prod_{i=1}^m \|f_i\|_{L^{p_i}(w_i^{p_i})}.
\end{eqnarray*}
This completes the proof.
\end{proof}

\section{Proof of Theorem~\ref{thm:m3}}
First, we introduce the sharp reverse H\"older's property of $A_\infty$ weights which was proved in \cite{HP} and \cite[Theorem 2.3]{HPR}. Recall that
\[
  [w]_{A_\infty}:=\sup_Q \frac{1}{w(Q)}\int_Q M(w\chi_Q).
\]
\begin{Proposition}[{\cite[Theorem 2.3]{HP}}]\label{prop:reverseH}
Let $w\in A_\infty$. Then
\begin{equation}\label{eq:e2}
\left(\frac{1}{|Q|}\int_Q w^{r(w)}\right)^{1/{r(w)}}\le 2\frac{1}{|Q|}\int_Q w,
\end{equation}
where $r(w)=1+\frac{1}{\tau_n [w]_{A_\infty}}$ and $\tau_n=2^{11+n}$. Notice that the conjugate $r(w)'\simeq [w]_{A_\infty}$.
\end{Proposition}
We also need the following characterization of $A_{\vec{P},q}$ weights.
\begin{Proposition}[{\cite[Theorem 3.4]{M}}]
Suppose that $1<p_1,\cdots,p_m<\infty$, and $\vec{w}\in A_{\vec{P},q}$. Then
\[
  u_{\vec w}\in A_{mq}\quad\mbox{and}\quad \sigma_i\in A_{mp_i'}.
\]
\end{Proposition}
\begin{proof}[Proof of Theorem~\ref{thm:m3}]
 Set $\alpha_i=(p_i'r_i)'$,
where $r_i$ is the exponent in the sharp reverse H\"older's inequality (\ref{eq:e2})
for the weights $\sigma_i$ which are in $A_\infty$ for $i = 1$, $\cdots$, $m$.
By (\ref{eq:e1}), we have
\begin{eqnarray*}
&&\int_{\bbR^n}\mathcal{M}_{\alpha}^{\mathscr{D}}(f_1,\cdots,f_m)^q
  u_{\vec w}dx\\
&\le&\sum_{k,j}\left(\prod_{i=1}^m\frac{1}{|Q_j^k|^{1-\alpha/{mn}}}
  \int_{Q_j^k} |f_i(y_i)| dy_i\right)^q u_{\vec w}(E(Q_j^k))\\
&\le&\sum_{k,j}\frac{u_{\vec w}(E(Q_j^k))}{|Q_j^k|^{-\alpha q/n}}
  \prod_{i=1}^m \left(\frac{1}{|Q_j^k|}\int_{Q_j^k}|f_i|^{\alpha_i}
     w_i^{\alpha_i}\right)^{\frac{q}{\alpha_i}}
\cdot\left(\frac{1}{|Q_j^k|}\int_{Q_j^k}w_i^{-\alpha_i'}
  \right)^{\frac{q}{\alpha_i'}}\\
&\le&\sum_{k,j}\frac{u_{\vec w}(E(Q_j^k))}{|Q_j^k|^{-\alpha q/n}}
   \prod_{i=1}^m \left(\frac{1}{|Q_j^k|}\int_{Q_j^k}|f_i|^{\alpha_i}
      w_i^{\alpha_i}\right)^{\frac{q}{\alpha_i}}
   \cdot \left(\frac{2}{|Q_j^k|}\int_{Q_j^k}\sigma_i\right)^{\frac{q}{p_i'}}\\
&&\hskip 70mm \mbox{(by Proposition~\ref{prop:reverseH})}    \\
&\le&C[\vec{w}]_{A_{\vec{P},q}}\sum_{k,j}|E(Q_j^k)|\prod_{i=1}^m
  \left(\frac{|Q_j^k|^{\frac{\alpha p \alpha_i}{n p_i}}}{|Q_j^k|}
    \int_{Q_j^k}|f_i|^{\alpha_i}w_i^{\alpha_i}\right)^{\frac{q}{\alpha_i}}\\
&\le&C[\vec{w}]_{A_{\vec{P},q}}\int_{\bbR^n}\prod_{i=1}^m
   M_{\frac{\alpha p \alpha_i}{p_i}}^{\mathscr{D}}
     (|f_i|^{\alpha_i} w_i^{\alpha_i})^{q/{\alpha_i}}dx\\
&\le&C[\vec{w}]_{A_{\vec{P},q}}\prod_{i=1}^m\left(\int_{\bbR^n}
   M_{\frac{\alpha p \alpha_i}{p_i}}^{\mathscr{D}}(|f_i|^{\alpha_i}
      w_i^{\alpha_i})^{q_i/{\alpha_i}}dx\right)^{q/{q_i}},
\end{eqnarray*}
where
\[
  \frac{1}{q_i}=\frac{1}{p_i}-\frac{\alpha p}{n p_i},\quad i=1,\cdots,m.
\]
Substituting $ \alpha p \alpha_i/p_i$ for $\alpha$ and letting $w=1$ in
Proposition~\ref{prop:p2}, we get
\begin{eqnarray*}
&&\int_{\bbR^n}\mathcal{M}_{\alpha}^{\mathscr{D}}(f_1,\cdots,f_m)^q
  u_{\vec{w}}dx\\
&\le& C[\vec{w}]_{A_{\vec{P},q}}\prod_{i=1}^m
  \left(1+\frac{(p_i/{\alpha_i})'}{(q_i/{\alpha_i})}
  \right)^{(1-\frac{\alpha p\alpha_i}{n p_i})
    \frac{q}{\alpha_i}}\bigg\||f_i|^{\alpha_i} w_i^{\alpha_i}
       \bigg\|_{L^{\frac{p_i}{\alpha_i}}(\bbR^n)}^{\frac{q}{\alpha_i}}\\
&=& C[\vec{w}]_{A_{\vec{P},q}}\prod_{i=1}^m \left(1+\frac{(p_i/{\alpha_i})'}
  {(q_i/{\alpha_i})}\right)^{(1-\frac{\alpha p\alpha_i}{n p_i})
    \frac{q}{\alpha_i}}\|f_i\|_{L^{p_i}(w_i^{p_i})}^q.
\end{eqnarray*}
It is easy to check that $\frac{p_i}{\alpha_i}-1\simeq [\sigma_i]_{A_\infty}^{-1}$.
Therefore
\[
  1+\frac{(p_i/{\alpha_i})'}{(q_i/{\alpha_i})}\lesssim 1+\frac{p_i}{q_i}[\sigma_i]_{A_\infty}\lesssim [\sigma_i]_{A_\infty}
\]
and
\begin{eqnarray*}
  (1-\frac{\alpha p\alpha_i}{n p_i})\frac{q}{\alpha_i}
&=&\frac{q}{p_i}-\frac{\alpha pq}{np_i}+\frac{q}{p_i}(\frac{p_i}{\alpha_i}-1) \\
&\le&
  \frac{q}{p_i} \left(1-\frac{\alpha p}{n}\right)
  +  \frac{q}{p_i}\cdot C [\sigma_i]_{A_\infty}^{-1}.
\end{eqnarray*}
Consequently,
\[
  \left(1+\frac{(p_i/{\alpha_i})'}{(q_i/{\alpha_i})}\right)^{(1-\frac{\alpha p\alpha_i}{n p_i})\frac{q}{\alpha_i}}
  \lesssim [\sigma_i]_{A_\infty}^{\frac{q}{p_i}(1-\frac{\alpha p}{n})}.
\]
Now we get
\[
  \|\mathcal{M}_{\alpha}(\vec{f})\|_{L^q(u_{\vec{w}})}\le
  C_{m,n,\vec{P},q}\left([\vec{w}]_{A_{\vec{P},q}}^{\frac{1}{q}}
  \prod_{i=1}^m [\sigma_i]_{A_\infty}^{\frac{1}{p_i}(1-\frac{\alpha p}{n})}\right)
  \prod_{i=1}^m \|f_i\|_{L^{p_i}(w_i^{p_i})}.
\]
\end{proof}

\section{Proof of Theorem \ref{thm:m4}}
To prove Theorem we will need the following Lemma about Carleson sequences (see \cite[Theorem 4.5]{HP} and \cite[Lemma 5.3]{CM2}).

\begin{Lemma} \label{lem:carleson}Suppose $a=\{a_Q\}_{Q\in \mathscr D}$ and $c=\{c_Q\}_{Q\in \mathscr D}$ are sequences with $c_Q$ nonnegative, and $\mu$ is a positive Borel measure.  Set $M^{\mathscr D} a(x)=\sup_{Q\ni x} |a_Q|$ and
\begin{equation}\label{carlconst}\mathscr C(c)=\sup_{R\in \mathscr D} \frac{1}{\mu(R)}\sum_{Q\subseteq R} c_Q.\end{equation}
If $r>0$ and $\mathscr C(c)<\infty$, then
$$\sum_{Q\in \mathscr D} |a_Q|^rc_Q\leq \mathscr C(c) \int_{\mathbb R^n} (M^{\mathscr D} a)^r\,d\mu.$$
\end{Lemma}

If $\mathscr C(c)<\infty$ in \eqref{carlconst} we say that $c=\{c_Q\}$ is a Carleson sequence with respect to $\mu$.   The constant $\mathscr C(c)$ is called the Carleson constant.  We are now ready to prove Theorem \ref{thm:m4}.

\begin{proof} Using the same decomposition as in the proof of Theorem \ref{thm:m3} we arrive at
\begin{align*}
\lefteqn{\left(\int_{\bbR^n}\mathcal{M}_{\alpha}^{\mathscr{D}}(f_1,\cdots,f_m)^q
  u\,dx\right)^{1/q}}\\
&\le\left(\sum_{k,j}\left(|Q_{j^k}|^{\frac{\alpha}{n}}\prod_{i=1}^m\frac{1}{|Q_j^k|}
  \int_{Q_j^k} |f_i(y_i)| dy_i\right)^q u(E(Q_j^k))\right)^{1/q}\\
  &\le [u,\vec w]_{A_{\vec{P},q}}\left(\sum_{k,j}\left(\prod_{i=1}^m
  \int_{Q_j^k} |f_i(y_i)| dy_i\,\sigma_i(Q_j^k)^{-\frac1{p_i'}}\right)^q\right)^{1/q}.
\end{align*}
Let $q_i=\frac{qp_i}{p}\geq p_i$ and note that $\frac{1}{q}=\frac{1}{q_1}+\cdots+\frac{1}{q_m}$ so we may use a discrete H\"older's inequality to obtain:
\begin{multline*}
\left(\sum_{k,j}\left(\prod_{i=1}^m
  \int_{Q_j^k} |f_i(y_i)| dy_i\,\sigma_i(Q_j^k)^{-\frac1{p_i'}}\right)^q\right)^{1/q}\\
 \leq \prod_{i=1}^m \left(\sum_{k,j}
  \left(\int_{Q_j^k} |f_i(y_i)| dy_i\,\sigma_i(Q_j^k)^{-\frac1{p_i'}}\right)^{q_i}\right)^{1/q_i}.
\end{multline*}
Next define $\beta_i$ by
$$\frac{\beta_i}{n}=\frac{1}{p_i}-\frac{1}{q_i}$$
and note that
$$
  \left(\int_{Q_j^k} |f_i(y_i)| dy_i\,\sigma_i(Q_j^k)^{-\frac1{p_i'}}\right)^{q_i}= \left(\frac{1}{\sigma_i(Q_i^k)^{1-\frac{\beta_i}{n}}}\int_{Q_j^k} |f_i(y_i)| dy_i\,\right)^{q_i}\sigma_i(Q_j^k).$$
We let
$$a_Q=\frac{1}{\sigma_i(Q)^{1-\frac{\beta_i}{n}}}\int_{Q} |f_i(y_i)| dy_i$$
if $Q=Q_j^k$ for some $j,k$ and $a_Q=0$ otherwise.  Likewise let $c_Q=\sigma_i(Q)$ if $Q=Q_j^k$ and $c_Q=0$ otherwise.  Further, notice that $\{c_Q\}$ is a Carleson sequence with respect to $\sigma_i$ with constant $\mathscr C(c)\leq [\sigma_i]_{A_\infty}$.  Indeed,
\begin{align*}
\sum_{Q\subset R} c_Q= \sum_{j,k: Q^k_j\subset R} \sigma_i(Q_j^k)&=\sum_{j,k: Q^k_j\subset R} \sigma_i(Q_j^k)\\
&\lesssim\sum_{j,k: Q^k_j\subset R} \frac{\sigma_i(Q_j^k)}{|Q_j^k|}|E(Q_j^k)|\\
&\leq\int_R M(\chi_R\sigma_i)\,dx\\
&\leq [\sigma_i]_{A_\infty} \sigma_i(R).
\end{align*}
Moreover
$$M^{\mathscr D} a(x)\leq M^{\mathscr D}_{\beta_i,\sigma_i}(f\sigma_i^{-1})(x)$$
and by Proposition \ref{prop:p2} we have
$$\|M^{\mathscr D}_{\beta_i,\sigma_i}(f\sigma_i^{-1})\|_{L^{q_i}(\sigma_i)}\lesssim \|f\sigma_i^{-1}\|_{L^{p_i}(\sigma_i)}=\|f\|_{L^{p_i}(w_i^{p_i})}.$$
Combining everything, we have
\begin{multline*}
{\left(\int_{\bbR^n}\mathcal{M}_{\alpha}^{\mathscr{D}}(f_1,\cdots,f_m)^q
  u\,dx\right)^{1/q}}\lesssim\\
 [u,\vec{w}]_{A_{\vec P,q}}\prod_{i=1}^m [\sigma_i]_{A_\infty}^{1/q_i}  \|f_i\|_{L^{p_i}(w_i^{p_i})}= \Big([u,\vec{w}]_{A_{\vec P,q}}\prod_{i=1}^m [\sigma_i]_{A_\infty}^{\frac{p}{qp_i}} \Big) \prod_{i=1}^m\|f_i\|_{L^{p_i}(w_i^{p_i})}.
\end{multline*}
\end{proof}

\section{Examples}
Finally, we end with some examples to show that our bounds are sharp.  First we show that Theorem~\ref{thm:m1} is sharp.  Consider the case $m=2$ (we leave it to the reader to modify the example for $m>2$) and suppose that
\begin{equation}\label{sharp}
\|\mathcal{M}_\alpha\|_{L^{p_1}(w_1^{p_1})\times L^{p_2}(w_2^{p_2}) \rightarrow L^q(w_1^q w_2^q)}\lesssim [\vec{w}]_{A_{\vec{P},q}}^{r(1-\frac{\alpha}{n})\max(\frac{p_1'}{q},\frac{p_2'}{q})}
\end{equation}
for some $r<1$.  Further suppose that $p_1'\geq p_2'$.  For $0<\ep<1$, let $f_1(x)=|x|^{\ep-n}\chi^{}_{B(0,1)}(x)$, $f_2(x)=|x|^{\frac{\ep-n}{p_2}}\chi_{B(0,1)}(x)$, $w_1(x)=|x|^{(n-\ep)/{p_1'}}$ and $w_2(x)=1$.  Calculations show that
$\|f_1\|_{L^{p_1}(w_1^{p_1})}\simeq\ep^{-{1}/{p_1}}, \|f_2\|_{L^{p_2}(w_2^{p_2})}\simeq\ep^{-1/p_2},$
and
$[\vec{w}]_{A_{\vec{P},q}}\simeq\ep^{-{q}/{p_1'}}.$
For $x\in B(0,1)$ we have
\begin{eqnarray*}
\mathcal{M}_\alpha(f_1,f_2)(x)&\gtrsim&\frac{1}{|x|^{n-\frac{\alpha}{2}}}\int_{B(0,|x|)} |y_1|^{\ep-n}\,dy_1 \cdot \frac{1}{|x|^{n-\frac{\alpha}{2}}}\int_{B(0,|x|)} |y_2|^{\frac{\ep-n}{p_2}}\,dy_2\\
&\gtrsim&\frac{1}{\ep}|x|^{\ep-n+\frac{\ep-n}{p_2}+\alpha}.
\end{eqnarray*}
Hence,
\begin{eqnarray}\label{eq:e3}
\|\mathcal{M}_\alpha(f_1,f_2)\|_{L^q(w_1^{q}w_2^q)}&\gtrsim &\frac{1}{\ep}\Big(\int_{B(0,1)}|x|^{(\ep-n)(q+\frac{q}{p_2}-\frac{q}{p_1'})+\alpha q}\,dx\Big)^{1/q}\\
&\simeq&\frac1\ep\Big(\int_0^1x^{(1+\frac{\alpha q}{n})\ep-1}\,dx\Big)^{1/q}\nonumber\\
&\simeq&\frac{1}{\ep}\Big(\frac{1}{\ep}\Big)^{1/q}.\nonumber
\end{eqnarray}
Combining this with inequality (\ref{sharp}) we see for some $r<1$,
$\Big(\frac1\ep\Big)^{1+\frac1q}\lesssim \Big(\frac1\ep\Big)^{r(1-\frac\alpha n)+\frac1p},$
which is impossible as $\ep\rightarrow 0$.

Next we show that Theorem~\ref{thm:m2} is sharp.
It is easy to notice that $\mathcal{M}_\alpha\le C_{m,n,\alpha} \mathcal{I}_\alpha$. If
  $\max_i p_i'\ge q$, then using the same $f_i$ and $w_i$ as above we get (\ref{eq:e3}) with $\mathcal{M}_\alpha$
replaced by $\mathcal{I}_\alpha$, showing the sharpness. On the other hand, if $\max_i p_i'< q$, the sharpness
follows from the standard duality argument used in the proof of Theorem~\ref{thm:m2}.

Finally we show that Theorem~\ref{thm:m3} is sharp.
For $0<\ep<1$, let
\[
  w_i(x)=|x|^{(n-\ep)/{p_i'}}\quad\mbox{and}\quad f_i(x)=|x|^{\ep-n}\chi^{}_{B(0,1)}(x), \quad i=1,\cdots,m.
\]
Then $u_{\vec{w}}=|x|^{(n-\ep)(m-\frac 1 p)q}$ and it is easy to check that
\[
  [\vec{w}]_{A_{\vec{p},q}}\simeq \left(\frac{1}{\ep}\right)^{q(m-\frac{1}{p})},\quad [\sigma_i]_{A_\infty}\lesssim \frac{1}{\ep}
  \quad\mbox{and}\quad \prod_{i=1}^m \|f_i\|_{L^{p_i}(w_i^{p_i})}\simeq \left(\frac{1}{\ep}\right)^{1/p}.
\]
For $x\in B(0,1)$, we have
\begin{eqnarray*}
\mathcal{M}_\alpha(\vec{f})(x)&\gtrsim&\prod_{i=1}^m\frac{1}{|x|^{n-\frac{\alpha}{m}}}\int_{B(0,|x|)} |y_i|^{\ep-n}\,dy_i\\
&\gtrsim&\left(\frac{1}{\ep}\right)^m |x|^{m(\ep-n)+\alpha}.
\end{eqnarray*}
Therefore,
\begin{eqnarray*}
\|\mathcal{M}_\alpha(\vec{f})\|_{L^q(u_{\vec{w}})}&\gtrsim& \left(\frac{1}{\ep}\right)^m \left(\int_{B(0,1)}|x|^{mq(\ep-n)+\alpha q +(n-\ep)(m-\frac{1}{p})q}dx\right)^{\frac{1}{q}}\\
&\simeq & \left(\frac{1}{\ep}\right)^m \left(\int_0^1 t^{(\alpha+1)\ep -1}dt\right)^{\frac{1}{q}}\\
&\simeq& \left(\frac{1}{\ep}\right)^{m+\frac{1}{q}}\\
&\gtrsim&\left([\vec{w}]_{A_{\vec{P},q}}^{\frac{1}{q}}\prod_{i=1}^m [\sigma_i]_{A_\infty}^{\frac{1}{p_i}(1-\frac{\alpha p}{n})}\right)\prod_{i=1}^m \|f_i\|_{L^{p_i}(w_i^{p_i})}.
\end{eqnarray*}


\end{document}